\documentclass[12pt]{article}

\usepackage[T2A]{fontenc}
\usepackage[utf8]{inputenc}
\usepackage[russian,ukrainian,english]{babel}

\usepackage{amssymb}
\usepackage{graphics}
\usepackage{mathtools} 

\newcommand{\qed}{$\;\;\;\Box$}
\newenvironment{proof}{\par\smallbreak{\sl Proof.~}}
{\unskip\nobreak\hfill \qed \par\medbreak}

\newcounter{claim}
\renewcommand{\theclaim}{\arabic{claim}}
{\par\medskip\par}

{\qed\par\smallbreak}
\newcommand{\hide}[1]{}

\newcommand{\N}{{\mathbb N}}

\newcommand{\R}{{\mathbb R}}
\newcommand{\CC}{{\mathbb C}}

\newcommand{\A}{{\cal A}}
\newcommand{\B}{{\cal B}}

\newcommand{\LL}{{\cal L}}

\newcommand{\beq}{\begin{equation}}
\newcommand{\ee}{\end{equation}}

\renewcommand{\d}{\partial}

\newtheorem{thm}{Theorem}[section]
\newtheorem{lemma}[thm]{Lemma}
\newtheorem{prop}[thm]{Proposition}
\newtheorem{defn}[thm]{Definition}
\newtheorem{cor}[thm]{Corollary}
\newtheorem{rem}[thm]{Remark}

\newcommand{\al}{\alpha}
\newcommand{\be}{\beta}
\newcommand{\ga}{\gamma}
\newcommand{\de}{\delta}
\newcommand{\eps}{\varepsilon}
\newcommand{\vphi}{\varphi}
\newcommand{\la}{\lambda}
\newcommand{\om}{\omega}
\newcommand{\io}{\iota}
\newcommand{\reff}[1]{(\ref{#1})}

\newcommand{\sgn}{\mathop{\rm sgn}}

\newcommand{\diag}{\mathop{\rm diag}\nolimits}

\renewcommand{\Re}{\mathop{\mathrm{Re}}\nolimits}

\title{Finite Time Stabilization of
	Nonautonomous \\ First Order Hyperbolic Systems
}

\newcounter{thesame}
\author{
	Irina Kmit
	\thanks{Institute of Mathematics, Humboldt University of Berlin. On leave from the
		Institute for Applied Problems of Mechanics and Mathematics,
		Ukrainian National Academy of Sciences. {\small   E-mail:
			{\tt kmit@mathematik.hu-berlin.de}}}
	\ \ \ Natalya Lyul'ko \thanks{Sobolev Institute of Mathematics, Russian Academy of Sciences and
		Novosibirsk State University, Russia.
		{\small   E-mail:
			{\tt natlyl@mail.ru}}
}}

\date{}

\begin{document}
	
	\maketitle

	\begin{abstract}
		\noindent
	We address nonautonomous  	  initial  boundary value problems
for decoupled linear first-order one-dimensional hyperbolic systems,
investigating the phenomenon of finite time stabilization.
We establish sufficient and necessary conditions 
ensuring  that  solutions stabilize to zero in a finite time
for any initial $L^2$-data.	In the nonautonomous case we give a   combinatorial  criterion stating that the robust stabilization occurs
if and only if  the matrix of  reflection boundary coefficients 
corresponds to a
directed acyclic graph. An equivalent robust algebraic criterion is that the 	adjacency matrix of this graph is  nilpotent. In the autonomous  case we also provide   a spectral 
stabilization criterion, which is nonrobust with respect to perturbations of  the coefficients of the hyperbolic system.
\end{abstract} 


\section{Problem setting  and main results}\label{sec:problem}
\renewcommand{\theequation}{{\thesection}.\arabic{equation}}
\setcounter{equation}{0}

The paper concerns the finite time stabilization property  
in the semistrip $\Pi=\{(x,t):\, 0\le x\le 1,\, 0\le t<\infty\}$
of  solutions to the $n\times n$-decoupled
first-order system
\begin{equation}\label{eq:1}
\partial_t u_j + a_j(x,t)\partial_x u_j + b_j(x,t)u_j = 0,\quad 0<x<1,\, t>0,\,
j\le n,
\end{equation}
endowed with the reflection boundary conditions
\begin{equation}\label{eq:3}
\begin{array}{l}
\displaystyle
u_j(0,t) = \sum\limits_{k=1}^mp_{jk}u_k(1,t)+\sum\limits_{k=m+1}^np_{jk}u_k(0,t),\quad t\ge0,\,
1\le j\le m,
\\
\displaystyle
u_j(1,t) =\sum\limits_{k=1}^mp_{jk}u_k(1,t)+\sum\limits_{k=m+1}^np_{jk}u_k(0,t), \quad t\ge0,\,
m< j\le n,
\end{array}
\end{equation}
and the initial conditions
\begin{eqnarray}
u_j(x,0) = \varphi_j(x), \quad 0\le x\le 1,\,j\le n.
\label{eq:in}
\end{eqnarray}
Here  $n\ge 2$ and $0\le m\le n$ are fixed integers.
The unknown function $u=(u_1,\ldots,u_n)$ and the initial function
$\vphi=(\vphi_1,\dots,\vphi_n)$ are vectors of real-valued functions.
The coefficients $a_j$ and $b_j$ are
real-valued functions and the $n\times n$-matrix $P=(p_{jk})$ of the reflection boundary coefficients
has real entries.
The functions $a_j$  are supposed to satisfy the following conditions:
\begin{equation}\label{eq:h1}
\begin{array}{ll}
\inf\left\{a_j(x,t)\,:\,(x,t)\in\Pi, 1\le j\le m\right\}\ge a,\\ [2mm]
\sup\left\{a_j(x,t)\,:\,(x,t)\in\Pi, m+1\le j\le n\right\}\le -a
\end{array}
\end{equation}
for some $a>0$.

The purpose of the paper is to identify a class of  boundary conditions of the type \reff{eq:3}
ensuring that all solutions to the problem \reff{eq:1}--\reff{eq:in} 
stabilize to zero in a finite time not depending on the
initial data. To this end,  we establish several stabilization criteria
in terms of the reflection boundary coefficients and the coefficients of the hyperbolic system
(irrespectively of the initial data). A robust combinatorial criterion will be expressed in terms of a directed 
graph $G_P$ associated with the matrix~$P$.
Robust algebraic criteria will be given in terms  of the adjacency matrix of  $G_P$ or in terms of the matrix $P$ itself.
Moreover, we generalize these results to the case of nonautonomous boundary conditions.
For  autonomous problems we also
give a  nonrobust criterion in terms of spectral properties of the infinitesimal generator of the semigroup generated by the
autonomous problem~\reff{eq:1}--\reff{eq:in}. 

We have chosen to work in the $L^2$-setting, where the existence of $L^2$-generalized solutions  is
proved in \cite{KmLyul}. This gives us the advantage that the stabilization
criteria established  in this paper apply as well to  solutions
of better regularity. It should be noted that they also remain to be true for  solutions of worth regularity. In particular, for the
strongly singular delta-wave solutions the stabilization phenomenon can  easily be shown to
follow from the smoothing property proved  in \cite[Theorem 4.5]{Km}.

For a Banach space  $X$, the $n$-th Cartesian power $X^n$ is considered to be a Banach space normed by
$$
\|u\|_{X^n}= \max_{j\le n} \|u_j\|_{X},
$$
where $u=(u_1,\dots,u_n)$ with each $u_j\in X$.
By $C_0^\infty([0,1])$ we denote a subspace of the vector space $C^\infty([0,1])$ of
functions with  support within $(0,1)$.

Suppose that the coefficients of  \reff{eq:1}  fulfill the following regularity assumptions:
\begin{equation}\label{ss44}
\begin{array}{l}
\mbox{The functions } a_j, b_{j} \mbox{ belong to } C^1(\Pi)   
\mbox{ and are bounded}\\
\mbox{in $\Pi$ together with their first order derivatives}.
\end{array}
\end{equation}
It should be noted that the boundedness assumption on 
$a_j$ and $b_{j}$
and the
$C^1$-assumption  on $b_j$ are not crucial for the results of the present paper. The former 
can be dropped without loss of generality, while  the latter can be weakend 
to $b_j\in C^{0,1}_{x,t}(\Pi)$ or  $b_j\in C^{1,0}_{x,t}(\Pi)$ accordingly to the solution concept given by Definition \ref{defn:sol}, or even to $b_j\in C(\Pi)$
if one would use the solution concept as in \cite[Definition A.1]{Coron2021}.
The assumptions \reff{ss44} are imposed to simplify the presentation
(in particular,  they are supposed in the relevant result in  \cite{KmLyul} that we cite 
as
Theorem \ref{evol} below). 

As it follows from \cite[Theorem 3.1]{ijdsde}, for any continuously differentiable
initial function $\vphi$
satisfying the zero-order and the first-order
compatibility conditions between \reff{eq:3} and \reff{eq:in}
(in particular, for $\vphi\in C_0^\infty([0,1])^n$), the problem
\reff{eq:1}--\reff{eq:in} has a unique classical  $C^1$-solution in $\Pi$.
We now introduce a notion of the $L^2$-generalized  solution, which is 
analogous to that introduced in 
\cite[\S 29]{vladimirov1971} for  initial-boundary value problems
for an equation of the hyperbolic type.
\begin{defn}\cite[Definition 4.3]{KmLyul}\label{defn:sol}
	Let  $\vphi\in L^2(0,1)^n$.
	A function $u\in C\left([0,\infty), L^2(0,1)\right)^n$ is called
	an {\it $L^2$-generalized  solution} to the problem \reff{eq:1}--\reff{eq:in}
	if for any sequence $\vphi^l\in C_0^\infty([0,1])^n$ with
	$\vphi^l\to\vphi$ in $L^2(0,1)^n$
	the sequence
	of classical  $C^1$-solutions $u^l(x,t)$ to the problem
	\reff{eq:1}--\reff{eq:in} with
	$\vphi$ replaced by $\vphi^l$  fulfills the convergence condition 
	$$
	\|u(\cdot,t)-u^l(\cdot,t)\|_{L^2(0,1)^n} \to_{l\to\infty} 0,
	$$
	uniformly in $t$ varying in the range $0\le t\le T$, for each $T>0$.
\end{defn}

The following theorem  is obtained via  extension by continuity. 
Generally speaking, the proof method is based on the classical result
\cite[Theorem 2 in Section V.8.2]{kant} stating that,
for given normed spaces  $ X $ and $ Y$ such that 
$ Y $ is complete, any linear continuous operator $Q_0 $ from 
$\Omega \subset X $ to $ Y $ admits a unique linear continuous extension~$Q$ to the closure $\overline{\Omega} $ of  $\Omega$, and 
$\|Q\| =  \|Q_0\|$.

\begin{thm}\cite[Theorem 2.3]{KmLyul}\label{evol}
	Suppose that the conditions \reff{eq:h1} and \reff{ss44} are 
	fulfilled.Then, given
	$\vphi\in L^2(0,1)^n$, there exists a unique
	$L^2$-generalized  solution $u$
	to the problem~(\ref{eq:1})--(\ref{eq:in}).
\end{thm}

\begin{defn}\cite[Definition 4.3]{KmLyul}\label{defn:stab}\rm
	Problem \reff{eq:1}--\reff{eq:in} is said to be   {\it Finite Time Stabilizable (FTS)} if
	there exists  a positive real $T_e$  such that
	for every $\varphi\in L^2(0,1)^n$ the $L^2$-generalized  solution
	to \reff{eq:1}--\reff{eq:in} is a constant zero function for all  $t>T_e$. The value of
	$T_e$ is called   the {\it finite time stabilization}.
	The infimum value of all  $T_e$ with the
	above property is called the {\it optimal  stabilization time} and is denoted by  $T_{opt}$.
\end{defn}

\begin{defn}\label{def:robust}\rm
	The problem \reff{eq:1}--\reff{eq:in} is {\rm robust FTS} if it is FTS for any
	$a_j$ and $b_j$  satisfying \reff{eq:h1} and \reff{ss44}.
\end{defn}

First we provide a spectral FTS criterion  for the {\it autonomous} version  of the problem
\reff{eq:1}--\reff{eq:in}, when $a_j(x,t)\equiv a_j(x)$
and $b_j(x,t)\equiv b_j(x)$. Introduce diagonal matrices $A(x)=\diag(a_1,\dots,a_n)$ and $B(x) = \diag(b_1,\dots,b_n)$ and
write down the problem \reff{eq:1}--\reff{eq:in}
as the abstract Cauchy problem in  $L^2(0,1)^n$ in the following form:
\begin{equation}\label{ns0}
\begin{array}{ll}
\displaystyle\frac{d}{dt}u(t)=\A u(t),  \quad
u(0)=\varphi\in L^2(0,1)^n,
\end{array}
\end{equation}
where the operator
$ \A: D(\A)\subset L^2(0,1)^n\mapsto L^2(0,1)^n$ is defined by
\beq\label{A}
\displaystyle \left(\A v\right)(x)=-A(x)\frac{d v}{d x} - B(x)v
\ee
and
\beq\label{AAA}
D(\A)=\{v\in L^2(0,1)^n\,:\,\d_xv\in L^2(0,1)^n,\,v_{out}=
Pv_{in}\}.
\ee
Here
$$
\begin{array}{cc}
v_{out}=(v_1(0),...,v_m(0),v_{m+1}(1),...,v_n(1)),\\ [2mm]
v_{in}=(v_1(1),...,v_m(1),v_{m+1}(0),...,v_n(0)).
\end{array}
$$
\begin{thm}\label{tm_2}
	The autonomous problem \reff{eq:1}--\reff{eq:in} is FTS
	if and only if the spectrum of the operator
	$\A$ is empty.
\end{thm}

It should be emphasized that the criterion stated in Theorem
\ref{tm_2}
is not robust (see Subsection \ref{exk1}),
which is disadvantagable from the viewpoint of applications. To provide robust criteria,
with the matrix $P$ we associate the following directed graph~$G_P$:
\begin{itemize}
	\item
	$\{1,\dots,n\}$ is the vertex set of $G_P$,
	\item
	two vertices $j$ and $k$ are connected   by the arrow $(j,k)$ in $G_P$  if and only if  $p_{jk}\ne 0$.
\end{itemize}
Let us recall some notions from graph theory (see, e.g., \cite{Harary}).
A graph is  {\it directed } if its vertices are connected by edges
having  directions from one vertex to the other.
Formally, a directed graph  $G$ on a vertex set $V$ is determined by 
its edge set $E\subseteq V^2$, where $(i,j)\in E$ is a directed edge
(arrow) from vertex $i$ to vertex $j$. Let $l\ge 1$.
A {\it cycle of length $l$} in $G$ is 
a  sequence of pairwise distinct vertices $(k_1,\dots, k_l)$ such that  $(k_s,k_{s+1})\in E$ for all $s< l$ and $(k_l,k_1)\in E$.
Two cycles $(k_1,\ldots,k_l)$ and $(k'_1,\ldots,k'_l)$
are considered to be equal if they can be obtained from each other
by a cyclic shift or, equivalently, if the sets of their arrows
$\{(k_1,k_2),\ldots,(k_l,k_1)\}$ and $\{(k'_1,k'_2),\ldots,(k'_l,k'_1)\}$
are equal.
An {\it acyclic directed graph} is a directed graph having no cycles. 
For a directed graph on the vertex set $\{1,\dots,n\}$, the {\it adjacency matrix} is the
$n\times n$-matrix $W=(w_{jk})$ such that $w_{jk}$ is one when there is an arrow from vertex $j$ to vertex $k$ and zero otherwise.

The criteria in Theorems \ref{tm_5},  \ref{prl}, \ref{prl1}, and \ref{ss22} below
are stated for the {\it nonatono\-mous}
system \reff{eq:1}. 
We begin with a combinatorial criterion.

\begin{thm}\label{tm_5}
	The problem \reff{eq:1}--\reff{eq:in} is robust FTS if and only if the directed graph $G_P$ 
	is  acyclic.
\end{thm}

The following well-known result \cite{Sachs} yields that the combinatorial criterion
is efficiently recognizable. At the same time, it provides  an algebraic criterion
of finite time stabilizability of our problem.

\begin{prop}\label{prl0} 
	Let $G$ be a directed graph with adjacency $n\times n$-matrix $W$. Then $G$ is
	acyclic if and only if
	$W$ is nilpotent, with $W^n=0$.
\end{prop}

In the following theorem we collect a number of robust algebraic criteria.

\begin{thm}\label{prl}
	Let $P=(p_{jk})$ be an $n\times n$-matrix with real entries and 
	$W$  be the matrix with entries $w_{jk}=\sgn|p_{jk}|$. Then the following statements are equivalent:
	
	$(\io)$
	the problem \reff{eq:1}--\reff{eq:in} is robust finite time stabilizable;
	
	$(\io\io)$ the products $w_{i_1i_2}w_{i_2i_3}\cdots w_{i_{n}i_{n+1}}$  equal  zero for all tuples
	$\left(i_1,\dots,i_{n+1}\right)\in\left\{1,\dots,n\right\}^{n+1}$;
	
	$(\io\io\io)$ all principal minors of  the matrix $W$ equal zero;
	
	$(\io v)$ the matrix $W$ is nilpotent, with
	$W^n=0$.
	
\end{thm}

\begin{cor}\label{cortime}
	Let $P=(p_{jk})$ be an  $n\times n$-matrix with real entries and 
	$W$  be the matrix with entries $w_{jk}=\sgn|p_{jk}|$.	
	Assume that $W$ is nilpotent and 
	let 	$k_0$ be the minimum value of $k\le n$ such that $W^k=0$.

	$(\io)$ 
	Let
	$$
	a_0=\inf\left\{|a_j(x,t)|\,:\,(x,t)\in\Pi,\,  j\le n\right\}.
	$$
	Then the optimal   stabilization time  
	admits an upper bound
	$$
	T_{opt}\le \frac{k_0}{a_0}.
	$$
	
	$(\io\io)$	 Assume that $a_j(x,t)\equiv a_j(x)$ do not depend on 
	$t$ for all $j\le n$.
	Let $I$ be the set of all tuples
	$\left(i_1,\dots,i_{k_0}\right)\in\left\{1,\dots,n\right\}^{k_0}$
	such that $w_{i_1i_2}w_{i_2i_3}\cdots w_{i_{k_0-1}i_{k_0}}\ne 0$.
	Then
	$$
	\begin{array}{rcl}
	T_{opt}&=&
	\displaystyle\max\left\{\int_0^1\frac{dx}{|a_{j}(x)|}\,:\,j\le n\right\}  \ \ \textrm{if } k_0=1, \\ [5mm]
	T_{opt}&=&T^*
	\ \ \textrm{if } k_0=2,3,\\  [2mm]
	T_{opt}&\le&T^*
	\ \  \textrm{if } k_0>3,
	\end{array}
	$$
	where
	$$
	T^*=\max\left\{\sum_{k=1}^{k_0}\int_0^1\frac{dx}{|a_{i_k}(x)|}\,:\,
	\left(i_1,\dots,i_{k_0}\right)\in I\right\}.
	$$
\end{cor}

Theorem \ref{prl} can be recast as follows.

\begin{thm}\label{prl1}
	Let $P=(p_{jk})$ be an $n\times n$-matrix with real entries and  $P_{abs}=(|p_{jk}|)$.
	Then the following statements are equivalent:
	
	$(\io)$
	the problem \reff{eq:1}--\reff{eq:in} is robust finite time stabilizable;
	
	$(\io\io)$ the products $p_{i_1i_2}p_{i_2i_3}\cdots p_{i_{n}i_{n+1}}$ are equal to zero for all tuples
	$\left(i_1,\dots,\allowbreak i_{n+1}\right)\in\left\{1,\dots,n\right\}^{n+1}$;
	
	$(\io\io\io)$ all principal minors of  the matrix $P$ equal zero;
	
	$(\io v)$ the matrix $P_{abs}$ is nilpotent, with
	$P_{abs}^n=0$.
\end{thm}

Our results can be extended to the case of nonautonomous boundary conditions as follows.

\begin{thm}\label{ss22}
	Let $W=(w_{jk})$ be a  constant zero-one $n\times n$-matrix.
	For every $q_{jk}\in C^1(\R_+)$, where 
	$j,k\le n$,
	the problem
	\reff{eq:1}--\reff{eq:in} with $p_{jk}=q_{jk}(t)w_{jk}$
	is
	robust finite time stabilizable  if and only if   the matrix $W$ fulfills one of the conditions $(\io\io)$--$(\io v)$ of Theorem \ref{prl}
	is satisfied.
\end{thm}

The paper is organized as follows. Sections \ref{sec:state}--\ref{related}  motivate our research
and describe potential applications. Extentions  to evolution
families and applications to nonlinear problems are discussed in 
Section \ref{sec:Examples}. There we also give
examples showing the non-robustness of the spectral criterion.
In  Section  \ref{sec:criteria} we
prove our main results in  Theorems  \ref{tm_2},  \ref{tm_5},    and \ref{prl}
stating, respectively,
the spectral, combinatorial, and algebraic  stabilization criteria.  
This section also contains the proofs of Corollary \ref{cortime} about the optimal
stabilization time and
Theorem \ref{ss22} addressing 
nonautonomous boundary conditions.

\section{Motivation and comments}\label{sec: mot}
\renewcommand{\theequation}{{\thesection}.\arabic{equation}}
\setcounter{equation}{0}

\subsection{Motivation  and related work}\label{sec:state}

The FTS notion is motivated by the physical question whether solutions to an asymptotically stable system reach an equilibrium point (see, e.g., \cite{bhat}).
Last years  systems with the FTS property attract more and more attention, first of all due to
applications.
In particular,  they are well suitable to design 
controllers
and, therefore, are intensively studied
in control and system engineering
\cite{corbas,Coron,gugat,lakra,udw05,per}.
Starting with the work of D. Russell \cite{russell}
in  control theory  for linear autonomous hyperbolic systems, much research is
devoted to finding   boundary
controls transferring the system from an arbitrary initial state to the zero state, see also \cite{CorNgu,eller,huoliv}.
In the present paper, instead of finding  boundary controls, we provide classes of first-order hyperbolic systems ensuring the
above property of the solution operator $\vphi\to u(\cdot,T)$, namely the property that $u(\cdot,T)\equiv 0$ for all
initial functions $\vphi$. 
The distinct feature of our systems
is that the evolution processes they describe are irreversible in time.
In the literature much attention is also paid to finding or estimating the optimal stabilization time $T_{opt}$, being of special interest for engineers.
In \cite{CorNguyen}, the number $T_{opt}$  is explicitly computed for autonomous one-side control systems, where $a_j(x,t)\equiv a_j(x)$ and $b_j(x,t)\equiv b_j(x)$.
This is also one 
of the themes in our paper, cf. Corollary \ref{cortime}.

The concept of FTS plays an important role also in the research 
on the  adaptive output-feedback stabilization
\cite{hegez,BerKrs} and  inverse problems
\cite{tikh18}.

Another motivating area is photoacoustic imaging
\cite{CoxKara,StepUhl}. Even basic photoacoustic tomography models demonstrate
mathematical properties which are
crucial for  reconstruction of photoacoustic wave fields and that
are closely related to FTS systems.

N. \"Eltysheva \cite{Elt}, who is the second author of the present paper,
identifies  a class of  autonomous linear first-order hyperbolic systems with FTS property; see also \cite{guo}. This approach is based on  spectral analysis. 

In the present paper, we give a comprehensive FTS
analysis of  initial-boundary value problems for
a class of nonautonomous  hyperbolic systems.
While  nonautonomous case is   less studied
than  the  autonomous one,
nonautonomous phenomena  occur in many physical situations 
\cite{crit,flo,hegez,man}.

We pay a special attention to the robustness issue, which
is  important in the application areas mentioned above. 
It is absolutely clear that solving PDE systems numerically  requires a certain robustness. Moreover, concrete applications  usually  involve
measurement data  known only approximately.
In the control theory, the  quality  of    control  algorithms  is    estimated,
in particular,   by   robustness  with  respect to perturbations.
The robust FTS property (or, more generally, the robust stability property \cite{corbas}) is important in designing robust finite time controllers for  dynamical systems under modeling uncertainty. 

We remark that the robustness concept for quasilinear problems is a more
complicated issue (see, e.g., \cite{CorNguyen}), and this is a topic of our ongoing work. Another interesting direction is 
investigation of robustness not only with respect to  coefficients of the
hyperbolic system but also wih respect to  boundary coefficients.

\subsection{Related stability concepts}\label{related}
Stability properties of a dynamical system are crucial 
for adequate description of physical phenomena.
The FTS is an important  instance of  more general  concept of asymptotic stability. The last concept is
suggested
by Lyapunov in 1892 \cite{lyapunov} and is used to describe  the behavior of systems
within  infinite time intervals.  
The FTS concept is used to deal with the
systems operating over  finite  time intervals.
More precisely,  the FTS systems form a
subclass of asymptotically stable systems characterized  by  superstability property
which is studied in \cite{bal105,cr13,rw95} in the autonomous case
and  in \cite{KmLyul} in the nonautonomous case.

Consider an abstract evolution equation 
\begin{equation}\label{ss00}
\frac{d}{dt}x(t)=\B(t)x(t), \quad x(t)\in X,
\end{equation}
on  a Banach space $X$, where $\B(t): X\to X$ 
for each $t\ge 0$   is a linear operator.

\begin{defn}
	The system \reff{ss00}
	is called  {\rm exponentially stable} if there exist positive reals
	$\ga$ and $M$ such that every solution $x(t)$
	satisfies the estimate
	\beq\label{cor*}
	\|x(t)\|\le M e^{-\gamma t} \|x(0)\|, \quad t\ge 0.
	\ee
\end{defn}
The exponential stability for  hyperbolic systems  
with linear and nonlinear boundary conditions
has been intensively investigated in the literature by different methods,
in particular, by  the characteristic method \cite{Green,KmLyul,Slemrod}, by the Lyapunov’s function
approach \cite{Coron,cor1,diagne},  by the delay-equation approach 
\cite{chitour,CorNguyen},
and the backstepping approach \cite{CorNgu,coron2021_1,CorVas,2017-02-27}.

A stronger stability property is stated in the next definition.
\begin{defn}
	The system \reff{ss00}
	is called {\rm superstable} if the estimate \reff{cor*}
	holds for every $\gamma> 0$ and some $M=M(\gamma)$.
\end{defn}

Roughly speaking,  all solutions to  superstable  systems decay faster than any exponential as $t\to \infty$.  

It turns out that the concepts of superstability and FTS are the same in the autonomous
case we consider. Note that this is not true, in general (see  \cite{cr13}). The following fact is proved in Section \ref{sec:last}.

\begin{thm}\label{s1}
	System \reff{ns0} is superstable if and only if it is FTS.
\end{thm}

We conclude this subsection with a remark about finite-dimensional spaces.
Let $\B(t)=\B: X\to X$  be an autonomous  linear operator on a Banach space $X$.
Then the FTS  property makes sense  only if $X$ is infinite dimensional, since otherwise the operator
$\B$ has a non-empty point spectrum. 
In the nonlinear autonomous finite dimensional case, the FTS 
is investigated in the recent paper \cite{Lopez}, see also references therein.

\subsection{Further comments} \label{sec:Examples}
\subsubsection{About extensions
	to evolution families}
\label{exk5} 
We formulate our results  for the  initial-boundary value 
problem \reff{eq:1}--\reff{eq:in} where the initial time, say $\tau$, is fixed (to be zero). The established FTS criteria  do not depend on the initial function $\vphi$. From the dynamical point of view, it is interesting
to know  whether these criteria do not also depend  on the initial time  $\tau$.
The question can naturally be answered  in terms of the evolution
families
generated by the 
problem \reff{eq:1}, \reff{eq:3} with the initial conditions
\begin{eqnarray}
u_j(x,\tau) = \varphi_j(x), \quad 0\le x\le 1,\,j\le n,
\label{eq:int}
\end{eqnarray}
where $\tau\ge 0$ is arbitrary fixed.
Existence of the evolution family on the space
$L^2(0,1)^n$ in the case of autonomous boundary conditions 
is proved in 
\cite[Theorem~2.3]{KmLyul}. Note that the last result can be  extended 
to nonautonomous boundary conditions with bounded coefficients.

Recall that (see \cite[Definition 5.3]{Paz}) a two-parameter family $\{U(t,\tau)\}_{t\ge\tau}$ 
of linear bounded operators on a Banach space $X$
is called an  {\it evolution family}
if, first, 
$U(\tau,\tau) =I$ and $U(t,s)U(s,\tau) =  U(t,\tau)$ for all
$t \ge s \ge\tau$ and, second,
the map $(t,\tau)\in\R^2\rightarrow U(t,\tau)\in\LL(X)$ is strongly continuous for all
$t\ge \tau$.  	
We say that an evolution family  $\{U(t,\tau)\}_{t\ge\tau\ge 0}$ is
generated by the problem  \reff{eq:1}, \reff{eq:3}, \reff{eq:int} 	on the space $X=L^2(0,1)^n$ if, for each initial function $\vphi\in X$, the function 
$U(t,\tau)\vphi$ is the $L^2$-generalized solution at time $t$.

Let us introduce a 
dynamical notion of FTS generalizing  \cite[Definition 1.1]{Coron2021}.	
\begin{defn}\label{U}	\rm
	An evolution family 	$\{U(t,\tau)\}_{t\ge\tau\ge 0}$ 
	on a Banach space $X$ is called
	FTS if
	there exist  positive reals $T$ and  $K$ such that $U(\tau+T,\tau)=0$
	for all $\tau\ge 0$  and
	$\|U(t,\tau)\|_{\LL(X)}\le K$ for all $t\ge\tau\ge 0$ with $t-\tau\le 1$.
	The infimum value of all  $T$ with the
	above property is called the optimal stabilization time.
\end{defn}

Compared to Definition \ref{defn:stab}, the dynamical notion of FTS
additionally	requires uniform stabilization with respect to 
$\tau$  and  uniform boundedness of the evolution family.

Note that the criteria given by Theorems \ref{tm_2},  \ref{tm_5},    and \ref{prl} do not depend on the initial time. Therefore, the uniform
stabilization property in Definition~\ref{U} is satisfied automatically
for
the evolution family $U(t,\tau)$ generated by the problem  \reff{eq:1}, \reff{eq:3}, \reff{eq:int}
on the space $L^2(0,1)^n$. The uniform boundedness property is true
whenever  the coefficients $a_j$ and $b_j$ as well as the  nonautonomous 
(if any) boundary coefficients $p_{jk}$ are bounded. This  easily follows from the 
uniform
stabilization  property and the
following exponential bound
obtained in  \cite[Section~4.2]{KmLyul}:
$$
\|U(t,\tau)\|_{\LL(L^2(0,1)^n)}\le K_0e^{C_0(t-\tau)}\quad\mbox{for all } t\ge\tau,
$$
where constants $K_0\ge 1$ and $C_0>0$ do not depend
on $\tau$ and $t$. Consequently, if all the coefficients in \reff{eq:1}
and  \reff{eq:3} are bounded, then the   criteria for FTS
in the sense of Definition \ref{defn:stab}
remain to be  true  in the sense of Definition \ref{U}.

On the other hand, if the coefficients are not bounded, then the 
evolution family $\{U(t,\tau)\}_{t\ge\tau\ge 0}$ is not necessarily FTS. 
This  follows from the following simple example (a similar example
for one hyperbolic equation with an integral boundary condition is
constructed in
\cite[Remark 1.2]{Coron2021}):
$$
\begin{array}{ll}
\partial_t u_1- \partial_x u_1 = 0,\quad  \partial_t u_2+\partial_x u_2 = 0,\\
u_1(1,t) =0,\quad u_2(0,t) =p(t)u_1(0,t),\\
u_1(x,\tau) =\vphi_1(x),\quad	u_2(x,\tau) =\vphi_2(x),
\end{array}
$$
where $p(t)$ is a smooth   function such that $p(t)=0$ if $t\in[2k,2k+1]$ and
$p(t)=t$ if $t\in[2k+\frac{5}{4},2k+\frac{7}{4}]$
for all $k\in\N$. 
Using the method of characteristics,
one can  easily show that the optimal stabilization time equals two.
Moreover, 	for $\vphi=(\vphi_1,\vphi_2)$ $=(1,1)$ we have
$$
\begin{array}{rcl}
\left[U(2k+3/2, 2k+1)\vphi\right]_1&=&
\left\{
\begin{array}{ll}
1 & \mbox{if }\,0\le x \le 1/2\\
0 & \mbox{if } \,1/2< x \le 1,
\end{array}
\right.
\\ [4mm]
\left[U(2k+3/2, 2k+1)\vphi\right]_2&=&
\left\{
\begin{array}{ll}
p(2k+3/2-x) & \mbox{if }\, 0\le x \le 1/2\\
1 & \mbox{if } \,1/2< x \le 1.
\end{array}
\right.
\end{array}
$$
Hence,	
$$
\begin{array}{cc}
\displaystyle	 \left\|U\left(2k+\frac{3}{2}, 2k+1\right)\right\|_{\LL(L^2(0,1)^2)}\ge \frac{1}{\left\|\vphi\right\|_{L^2(0,1)^2}} \left\|U\left(2k+\frac{3}{2}, 2k+1\right)\vphi\right\|_{L^2(0,1)^2}\\ \displaystyle
\ge \frac{1}{\sqrt{2}}\left(\int_0^{1/2} \left|p\left(2k+\frac{3}{2}-x\right)\right|^2\,dx\right)^{1/2}
\\ \displaystyle\ge
\frac{1}{\sqrt{2}}\left(\int_0^{1/4} \left|p\left(2k+\frac{3}{2}-x\right)\right|^2\,dx\right)^{1/2}
\ge\frac{1}{2\sqrt{2}}\left(2k+\frac{5}{4}\right).
\end{array}
$$
It follows that, for any given $K>0$, there exist $\tau\ge 0$ and $t\ge \tau$
such that  $\|U(t,\tau)\|_{\LL(L^2(0,1)^2)}\ge K$, 
as claimed.

\subsubsection{ About the nonrobustness of the spectral criterion}
\label{exk1}
Consider  the $2\times 2$-system
\beq\label{nn66}
\d_tu_1+\d_xu_1=0,\quad \d_tu_2-\d_xu_2=0,
\ee
subjected to the boundary conditions
\beq\label{lii2}
u_1(0,t)=u_1(1,t)-u_2(0,t),\quad u_2(1,t)=u_1(1,t)-u_2(0,t)
\ee
and the initial conditions
\beq\label{eq:in1}
u_1(x,0)=\varphi_1(x),\quad u_2(x,0)=\varphi_2(x).
\ee
Since
$$
\det(I_2 - diag (e^{-\lambda},e^{-\lambda })P)=1-e^{-2\lambda}+e^{-2\lambda},
$$
the characteristic equation (see \reff{nn4} below) for  the operator  $\A$ generated by \reff {nn66}--\reff{lii2}
and defined by \reff{A}--\reff{AAA}
reads
$
1=0.
$
This means that $\A$  has empty spectrum and, by Theorem \ref{tm_2}, the problem
(\ref{nn66}), \reff{lii2}, \reff{eq:in1}  is FTS.

Next, consider the perturbed problem, with $\varepsilon$-perturbations in the leading part of the
differential  system. Specifically, we consider the system
\beq\label{nn77}
\d_tu_1+(1+\varepsilon)\d_xu_1=0,\quad \d_tu_2-\d_xu_2=0,
\ee
endowed with the boundary conditions \reff{lii2} and the initial conditions \reff{eq:in1}.
In this case
$$
\det(I_2 - diag (e^{\frac{-\lambda}{1+\varepsilon}},e^{-\lambda })P)=1+e^{-\lambda}-e^{\frac{-\lambda}{1+\varepsilon}},
$$
and the characteristic equation is
$$
1+e^{-\lambda}-e^{\frac{-\lambda}{1+\varepsilon}}=0.
$$
Consequently, the operator  $\A$ generated by the perturbed problem has infinitely many eigenvalues and, therefore,
the problem (\ref{nn77}), \reff{lii2},  \reff{eq:in1} is not FTS.

Finally, consider $\varepsilon$-perturbations in the lower-order part of the hyperbolic system, namely
\beq\label{nn771}
\d_tu_1+\d_xu_1+\varepsilon u_1=0,\qquad \d_tu_2-\d_xu_2=0,
\ee
endowed  with the conditions \reff{lii2} and  \reff{eq:in1}.
The characteristic equation  here reads
$$
1+e^{-\lambda}-e^{-(\lambda+\varepsilon)}=0.
$$
Again, the operator  $\A$ generated by this  problem has infinitely many eigenvalues and, therefore,
the problem (\ref{nn771}), \reff{lii2},  \reff{eq:in1}  is not FTS (even under small perturbations of 
$a_j$ or~$b_j$).

Remark that the matrix $P$  in the boundary conditions \reff{lii2}
does not satisfy the conditions of  the robust FTS Theorem  \ref{prl}. Indeed,
$$
P=\left(
\begin{array}{cc}
1&-1\\
1&-1
\end{array}
\right),\quad W=\left(
\begin{array}{cc}
1&1\\
1&1
\end{array}
\right),
$$
and, hence,  $W^2 \neq 0$.

\subsubsection{About applications to nonlinear problems}
\label{exk4} 
In \cite[Examples 3.4--3.6]{KmLyul}  we discuss problems from chemical kinetics
and boundary control theory modeled by nonlinear hyperbolic initial boundary value problems.
Their
linearizations at stationary solutions include boundary conditions fulfilling the assumptions of Theorem \ref{prl}. This makes the linearized problems to be perturbations of FTS problems, to which
one can  apply \cite[Theorem 2.7]{KmLyul}  about stability properties of the perturbed  problems.
This, in its turn, allows one to prove the asymptotic stability of stationary solutions
to the original nonlinear problems.

\section{Stabilization criteria}\label{sec:criteria}
\renewcommand{\theequation}{{\thesection}.\arabic{equation}}
\setcounter{equation}{0}

\subsection{Spectral criterion: Proof of Theorem \ref{tm_2}}
\label{sec: spectr}

Here we consider the autonomous version of our
problem  in the abstract form \reff{ns0}.

\subsubsection{Preliminary statements}\label{sec:prelim}
\vskip2mm

By means of
the (non-degenerate) change of variables $u_j\mapsto w_j$ where 
$$
w_j(x,t)=u_j(x,t)\exp\int_0^x\frac{b_j(\xi)}{a_j(\xi)}\,d\xi, \quad j\le n,
$$
we rewrite the problem \reff{ns0}
in the following equivalent form:
\beq\label{kk2}
\partial_t w  + A(x)\partial_x w = 0, \quad (x,t)\in \Pi,
\ee
\beq\label{kk3}
w_{out}(t)=P_1 w_{in}(t), \quad t\ge 0,
\ee
where $w= (w_1,\dots,w_n)$ and
\beq\label{nn2}\displaystyle
\begin{array}{ll}
	P_1=\diag(\underbrace{1,...,1}_m,e^{\beta_{m+1}},...,e^{\beta_n})P
	\diag(e^{-\beta_1},...,e^{-\beta_m},\underbrace{1,...,1}_{n-m}), \\
	\displaystyle  \beta_j=\int_0^1\frac{b_{j}(\xi)}{a_{j}(\xi)}\,d\xi.
\end{array}\ee
Note  that the problem  \reff{ns0} is FTS if and only if the problem (\ref{kk2}), (\ref{kk3})
with the initial conditions
$$
w_j(x,0)=\vphi_j(x)\exp\int_0^x\frac{b_j(\xi)}{a_j(\xi)}\,d\xi,\quad x\in[0,1],\quad j\le n,
$$
is FTS.

Switching now to an abstract setting, we proceed with the problem
\begin{equation}\label{ss8}
\begin{array}{rcl}\displaystyle
\frac{d}{dt}u(t)=\A_0u(t),  \quad
u(0)=\varphi\in L^2(0,1)^n,
\end{array}
\end{equation}
where the operator $\A_0 :  D(\A_0)\subset L^2(0,1)^n\mapsto L^2(0,1)^n$ is defined by
\beq\label{A0}
\left(\A_0v\right)(x)=-A(x)\frac{d v}{d x}
\ee
and 
\beq\label{sc10}
D(\A_0)=\{v\in L^2(0,1)^n\,:\,\d_xv\in L^2(0,1)^n,\,v_{out}=P_0v_{in}\},
\ee
the matrix  $P_0$ being an arbitrary $n\times n$-matrix with constant entries. Note
that  $\sigma(\A)=\sigma(\A_0)$
where operator $\A_0:  D(\A_0)\subset L^2(0,1)^n\mapsto L^2(0,1)^n$ is defined by  the right hand sides of \reff{A0}--\reff{sc10} with $P_0$ 
replaced by  $P_1$ as in (\ref{nn2}).

It is known from \cite[Theorem 6.29]{Kato} that the operator $\A_0$ is closed and
has finitely or countably many  eigenvalues, each of finite multiplicity.
Furthermore,  $\sigma(\A_0)$ has no finite limit
points.

In \cite[Lemma 4.2] {KmLyul} we proved the following apriori estimate for
the $L^2$-generalized solution  $u$  to the problem \reff{ss8}:
$$
\|u(\cdot,t)\|_{L_2(0,1)^n} \le K_0 e^{C_0t} \|\vphi\|_{L_2(0,1)^n},
\quad t\ge 0,
$$
the positive constants  $K_0$ and $C_0$ being independent of
$t$ and $\vphi$. It follows that the spectrum of
$\A_0$ lies in the semistrip $\Re \lambda\le C_0$.
Denote
\beq\label{nn3}
\al_j(x)=-\int_0^x\frac{1}{a_j(\xi)}\,d\xi\quad \mbox{and} \quad\tau_j=|\al_j(1)|\quad
\mbox{for all } j\le n.
\ee

Set
\beq\label{nn5}
\Delta(\lambda)=\det\left(I_n - \diag \left(e^{-\lambda \tau_1},....,
e^{-\lambda \tau_n}\right)P_0\right),
\ee
where $I_n$ is the unit $n\times n$-matrix.

\begin{lemma}\label{lem_1} A complex number $\lambda$
	is an eigenvalue of the operator $\A_0$ if and only if
	$\lambda$ satisfies the characteristic equation
	\beq\label{nn4}
	\Delta(\lambda)=0.
	\ee
\end{lemma}
\begin{proof}
	By definition, $\lambda \in \sigma(\A_0)$ if and only if
	there exists a nonzero function
	$y(x,\la)=(y_1(x,\la),\dots,y_n(x,\la))$ in $D(\A_0)$ fulfilling the equation $\lambda y= \A_0 y$ or, the same, the equation
	\beq\label{nn611}
	\lambda y= -A(x)\frac{d y}{d x}.
	\ee
	The general solution to \reff{nn611} is given by the formula
	\beq\label{nn612}
	y_j(x,\lambda)=c_j e^{\lambda\al_j(x)}, \quad j\le n,
	\ee
	where $c_j$  are arbitrary reals. To determine $c=(c_1,\dots,c_n)^T$,
	we use the boundary conditions $y_{out}=P_0 y_{in}$, which gives us
	the equation
	\beq\label{nn7}
	X(\lambda)c=0,
	\ee
	where
	\beq\label{nn8}
	\begin{array}{ll}
		X(\lambda)=\diag\biggl(\underbrace{1,\dots,1}_m,e^{\lambda\al_{m+1}(1)},\dots,e^{\lambda\al_{n}(1)}\biggr)\\
		\quad
		-P_0
		\diag\biggl(e^{\lambda\al_1(1)},\dots,e^{\lambda\al_{m}(1)},\underbrace{1,\dots,1}_{n-m}\biggr)\\
		\quad=\diag\biggl(e^{-\lambda\al_1(1)},\dots,e^{-\lambda\al_{m}(1)},e^{\lambda\al_{m+1}(1)},\dots,e^{\lambda\al_{n}(1)}\biggr)\\ 	\quad\times
		\biggl(I- diag \left(e^{-\lambda \tau_1},\dots,
		e^{-\lambda \tau_n}\right)
		P_0\biggr)\diag\biggl(e^{\lambda\al_1(1)},\dots,e^{\lambda\al_{m}(1)},\underbrace{1,\dots,1}_{n-m}  \biggr).
	\end{array} 
	\ee
	It follows that
	\beq\label{nn881}
	\det (X(\lambda))=\Delta(\lambda)\exp\Biggl\{\lambda\sum_{j=m+1}^n\al_j(1)\Biggr\}.
	\ee
	Therefore, the equation (\ref{nn7})
	has a nonzero solution if and only if  $\lambda$
	is a solution to (\ref{nn4}), as desired.
\end{proof}

\begin{rem}\label{rem_1}\rm The function $\Delta(\lambda)$ given by
	(\ref{nn5}) is a Dirichlet polynomial. Specifically,
	$$
	\Delta(\lambda)=1+\sum_{k=1}^M E_k e^{-\lambda r_k},
	$$
	where $M\ge 1$ is an integer, $r_1<r_2<...<r_M$ are  reals expressed in terms of
	$\tau_k$, and $E_k$ are reals expressed in terms of
	the entries of the matrix $P_0$ as well as the coefficients of the hyperbolic system.
	
	It is known from \cite[p. 266--268]{levin} that if at least one of the coefficients $E_k$   is nonzero, then  $\Delta(\lambda)$ has a countable number of zeros, lying in a strip 
	which is parallel to the imaginary axis. Consequently, the equation $\Delta(\lambda)=0$
	has no solutions in the complex plane if and only if $\Delta(\lambda)\equiv 1$
	or, the same, if and only if 
	$
	E_k=0
	$ 
	for all  $k\le M$.

\end{rem}

We conclude this subsection with a technical lemma.

\begin{lemma}\label{lem_2}
	Let $r\in C^1([0,1])$ 
	and let $q(x,\xi)$ be a continuously differentiable function in~$x\in [0,1]$
	and two times continuously differentiable in  $\xi\in [0,1]$.
	Suppose that $\partial_\xi q(x,\xi)\neq 0$ and there exists
	$q_0\ge 0$ such that
	\beq\label{nn51}
	q(x,\xi)\ge-q_0
	\ee
	for all $x, \xi\in [0,1]$. Then for all $t>q_0$ and $\gamma>0$ it holds
	\beq\label{nn52}
	\frac{1}{2\pi i}\int_{\gamma-i\infty}^{\gamma+i\infty} e^{\lambda t} \int_0^1 e^{\lambda q(x,\xi)}r(\xi)\,d\xi d\lambda\equiv 0.
	\ee
\end{lemma}

\begin{proof} Denote the left hand side of  (\ref{nn52}) by $I(x,t)$.
	Integrating by parts, we get
	\beq\label{nn53}
	I(x,t)=\frac{1}{2\pi i}\int_{\gamma-i\infty}^{\gamma+i\infty} \frac{e^{\lambda t}}{\lambda}\left(\frac{e^{\lambda q(x,\xi)}r(\xi)}{\partial_\xi q(x,\xi)} \Big|_{\xi=0}^{\xi=1}-\int_0^1 e^{\lambda q(x,\xi)}\partial_\xi\biggl(\frac{r(\xi)}{\partial_\xi q(x,\xi)}\biggr)\,d\xi\right)\, d\lambda.
	\ee
	
	It is known that for all $\gamma>0$ 
	$$
	\frac{1}{2\pi i}\int_{\gamma-i\infty}^{\gamma+i\infty} \frac{e^{\lambda (t-t_0)}}{\lambda}\,d\lambda=
	\left\{
	\begin{array}{rl}
	1 & \mbox{if}\,\, t>t_0\\
	0 & \mbox{if} \,\,t<t_0.
	\end{array}
	\right.
	$$
	Introduce a function
	$$g(t,x,\xi)= \frac{1}{2\pi i}\int_{\gamma-i\infty}^{\gamma+i\infty}  \frac{e^{\lambda (t+q(x,\xi))}}{\lambda}\, d\lambda=
	\left\{
	\begin{array}{rl}
	1 & \mbox{if}\,\,t+q(x,\xi)>0\\
	0& \mbox{if}\,\,t+q(x,\xi)<0.
	\end{array}
	\right.
	$$
	Suppose that  $t\ge q_0+\delta$ for some $\de>0$.
	By (\ref{nn51}), we have $g(t,x,\xi)\equiv 1$ for all  $x,\xi \in [0,1]$.  By the Dirichlet criterion of the uniform convergence of improper integrals,  the integral
	$$
	\frac{1}{2\pi i}\int_{\gamma-i\infty}^{\gamma+i\infty} \frac{e^{\lambda(t+ q(x,\xi))}}{\lambda}\partial_\xi\left(\frac{r(\xi)}{\partial_\xi q(x,\xi)}\right)\,d\lambda
	$$
	converges uniformly in $\xi\in [0,1]$ and, moreover, equals $\partial_\xi\left(\frac{r(\xi)}{\partial_\xi q(x,\xi)}\right)$. Changing
	the order of integration in the right hand side of \reff{nn53},  we get
	$$
	I(x,t)= \frac{r(\xi)}{\partial_\xi q(x,\xi)} \Big|_{\xi=0}^{\xi=1}-\int_0^1 \partial_\xi\left(\frac{r(\xi)}{\partial_\xi q(x,\xi)}\right)\,d\xi\equiv 0,
	$$
	as desired.
\end{proof}

\subsubsection{Proof of Theorem \ref{tm_2}}

As it follows from Section \ref{sec:prelim}, it suffices  to prove that the problem \reff{ss8} is FTS
if and only if $\sigma(\A_0)=\varnothing$.

\textit{Necessity.}  Let the problem \reff{ss8} be FTS.
If $\lambda\in \sigma(\A_0)$, then, on the account of \reff{nn612},
the function $[ u(t)](x)=u(x,t)=e^{\lambda t} y(x, \lambda)$ is a solution
to the problem \reff{ss8} with $\vphi(x)=y(x, \lambda)$, 
that is not equal to zero  for all  $t\ge 0$.
This contradicts to   the assumption that \reff{ss8}
is FTS. Therefore, $\sigma(\A_0)=\varnothing$.

\textit{Sufficiency.} Suppose that $\sigma(\A_0)=\varnothing$.
Due to Definition \ref{defn:sol}, it suffices to prove that there is $T>0$ such that, given
$\vphi \in C^\infty_0([0,1])^n$, the corresponding
continuously differentiable solution to the problem \reff{ss8}
equals zero for $t>T$.

Fix an arbitrary  $\vphi \in C^\infty_0([0,1])^n$.
Let $u(t)$ be the continuously differentiable solution to the problem
\reff{ss8} in  $\Pi$. We use the following apriori estimate (see,  e.g., \cite[Estimate (6)]{Lyu}):
$$
\|u(t)\|_{C^1([0,1])^n} \le K_1 e^{C_1t} \|\vphi\|_{C^1([0,1])^n}, \quad t\ge 0, 
$$
with constants $K_1$ and $C_1$ not depending on  $t$ and $\vphi$.
This estimate allows us to apply to~\reff{ss8} the Laplace transform in $t$, with the parameter $\lambda$. In the new unknown 
$$
\tilde{u}(x,\lambda)=\int_0^\infty e^{-\lambda t} u(x,t) \, dt, \quad t>0,\, \Re\la> C_1,
$$
the problem \reff{ss8} reads
\beq\label{nn42}
\lambda\tilde{u}-\A_0 \tilde{u}= \vphi(x).
\ee
Now we solve \reff{nn42} and show that
the solution
\beq\label{nn43}
u(x,t)=\frac{1}{2\pi i}\int_{\gamma-i\infty}^{\gamma+i\infty} e^{\lambda t} \tilde{u}(x,\lambda) \, d\lambda, \quad t>0,\quad \gamma>C_1,
\ee
to the problem \reff{ss8} becomes zero for $t>T$
for some $T\ge 0$ to be specified below.

A general solution to the equation  (\ref{nn42})
is given by the formula
\beq\label{nn44}
\tilde{u}(x,\lambda)=z(x,\lambda)+y(x,\lambda),
\ee
where for
a partial solution $z=(z_1,\dots,z_n)$ to the nonhomogeneous equation (\ref{nn42})
we have the representation
\beq\label{nn45}
z_j(x,\lambda)=\int_0^x
\exp\left\{-\lambda \int_\xi^x \frac{d\tau}{a_j(\tau)}\right\} \frac{\varphi_j(\xi)}{a_j(\xi)} \,d\xi,\quad j\le n,
\ee  while  $y(x,\lambda)$  is a general solution to  the homogeneous equation
(\ref{nn611}) and therefore is defined  by the formula (\ref{nn612}).
We  now determine $c$ in (\ref{nn612}). For that we use the fact that the function
$\tilde{u}(x,\lambda)$ fulfills the boundary conditions  $\tilde{u}_{out}=P_0\tilde{u}_{in}$ and get
\beq\label{nn46}
X(\lambda)c=b,
\ee
where $$b=P_0(z_1(1,\lambda),\dots,z_m(1,\lambda),\underbrace{0,...,0}_{n-m})^T-(\underbrace{0,\dots,0}_{m}, z_{m+1}(1,\lambda),\dots,z_n(1,\lambda))^T.
$$
Since $\sigma(\A_0)=\varnothing$, the function $X(\lambda)$ is invertible for all complex
numbers $\la$. From Lemma \ref{lem_1} and Remark \ref{rem_1} it follows that
$\Delta(\lambda)\equiv 1$. Therefore, the representation of  $\det(X(\lambda))$
given by (\ref{nn881}) reads
$$
\det(X(\lambda))=\exp\biggl\{\lambda\sum_{j=m+1}^n\al_j(1)\biggr\}.
$$
Set $X^{-1}(\lambda)= \{\kappa_{ij}(\lambda)\}_{i,j=1}^n$ and find
$\kappa_{ij}$. On the account of the representation (\ref{nn8}) for the matrix   $X(\lambda)$ we conclude that $\kappa_{ij}(\lambda)$ are entire functions
of $\lambda$ of the type 
$
\kappa_{ij}(\lambda)=\sum_{k=1}^{m_{ij}} \gamma_k e^{\lambda \mu_k},
$
where the reals $\mu_k$ 
are determined by  $\al_j(1)$, $j\le n$, while
the reals $\gamma_k$ are determined by  the elements of the matrix $P_0$.
Moreover, there exists a positive real $\mu^{*}$ such that
for all $k$ 
\beq\label{nn47}
-\mu^{*}\le \mu_k \le  \mu^{*}.
\ee
Taking into account the equation (\ref{nn46}) which reads  $c=X^{-1}(\lambda)b$,   for all $j\le n$ 
we obtain the following formula:
$$
c_j=\sum_{k=1}^m \kappa_{jk}(\lambda)\sum_{i=1}^m p^0_{ki}z_i(1,\lambda)+\sum_{k=m+1}^n \kappa_{jk}(\lambda)\left(\sum_{i=1}^m p^0_{ki}z_i(1,\lambda)-z_k(1,\lambda)\right),
$$
where $p^0_{jk}$ are the entries of the matrix $P_0$.
Therefore, for every $j\le n$ the function $y_j(x,\lambda)$
is a linear combination of entire functions in $\lambda$,
of the following type:
\beq\label{nn50}
\int_0^1 \exp\left\{\lambda\left( \alpha_j(x)-\int_\xi^1 \frac{d\tau}{a_l(\tau)}+\mu_k\right)\right\}\frac{\varphi_l(\xi)}{a_l(\xi)}  \,d\xi,
\ee
where $l\in\{1,...,n\}$ and $\mu_k$ satisfy \reff{nn47}.

Next we apply Lemma \ref{lem_2} to the functions $y(x,\lambda)$
and $z(x,\lambda)$ in (\ref{nn44}). On the account of (\ref{nn45}) and (\ref{nn50}), the function $u(x,t)$ given by (\ref{nn43})
equals zero for $t>T$, where
$$
T=\max _{ j, l\le n}\max _{\mu_k}\max _{x,\xi\in [01]}\left(\int_\xi^1\frac{d\tau}{a_l(\tau)}-\alpha_j(x)-\mu_k, \int_\xi^x\frac{d\tau}{a_j(\tau)}\right).
$$

Since $\sigma(\A_0)=\varnothing$, the problem \reff{ss8} is FTS.
Theorem \ref{tm_2} is therefore proved.

\begin{cor} \label{sll1}  \rm If the problem  
	\reff{eq:1}--\reff{eq:in}
	is robust FTS, then
	all principal minors of the matrix $P$  are  equal to zero.
\end{cor}
\begin{proof} 
	Due to Definition \ref{def:robust}, it  suffices  to prove the desired statement for  a partial case of 
	the problem  
	\reff{eq:1}--\reff{eq:in}, namely	for  	\reff{ss8} (autonomous version of 	\reff{eq:1}--\reff{eq:in} with $b_j\equiv 0$).
	Assume that	\reff{ss8} is robust FTS and prove that
	all principal minors of the matrix $P_0$   equal zero.
	
	Let  $(P_0)_{j_1j_2\dots j_l}$ be the principal minor of the matrix $P_0$, i.e. the determinant of the restriction of $P_0$
	to the rows and columns with indices $j_1,j_2,\dots,j_l$.	By Theorem~\ref{tm_2}, the problem \reff{ss8} is FTS iff the characteristic equation  $\Delta(\lambda)=0$ has no solutions. The function   $\Delta(\lambda)$ can be expressed as follows:
	\beq\label{nn90}
	\begin{array}{cc}
		\Delta(\lambda)=\det\left(I_n - \diag (e^{-\lambda \tau_1},\dots,e^{-\lambda \tau_n})P_0\right)=
		\\
		\displaystyle 1+\sum_{l=1}^n (-1)^l\sum_{1\le j_1< j_2<\dots< j_l\le n}e^{-\lambda(\tau_{j_1}+\tau_{j_2}+\dots+\tau_{j_l})} (P_0)_{j_1j_2\dots j_l}.
	\end{array}
	\ee
	Indeed, recall the following well-known formula for the  characteristic polynomial of the matrix $P_0$:
	\beq\label{char}
	\det(\lambda I_n-P_0)=\lambda^n-\lambda^{n-1}\sum_{j=1}^n (P_0)_{j}+\sum_{k=2}^n (-1)^k \lambda^{n-k}\sum_{1\le j_1< j_2<\dots< j_k\le n} (P_0)_{j_1j_2\dots j_k}.
	\ee
	To derive \reff{nn90}, we apply the same argument as for deriving
	\reff{char} 
	and, therefore, get
	$$
	\begin{array}{ll}
	\displaystyle\Delta(\lambda)=\det(I_n - \diag (e^{-\lambda \tau_1},\dots,e^{-\lambda \tau_n})P_0)\\ [2mm]=\displaystyle
	e^{-\lambda\sum_{i=1}^n\tau_i}\det(\diag (e^{\lambda \tau_1},\dots,e^{\lambda \tau_n})I_n - P_0)\\
	=\displaystyle
	1-\sum_{j=1}^n e^{-\lambda \tau_j}(P_0)_{j}
	+\sum_{k=2}^n (-1)^k\sum_{1\le j_1< j_2<\dots< j_k\le n}e^{-\lambda(\tau_{j_1}+\tau_{j_2}+\dots+\tau_{j_k})} (P_0)_{j_1j_2\dots j_k}.
	\end{array}
	$$
	The formula \reff{nn90} now easily follows.
	
	By assumption, the problem \reff{ss8} is FTS for any matrix $A$. Hence, due to Theorem~\ref{tm_2},
	for any  fixed $P_0$  and all positive reals
	$\tau_1, \dots,\tau_n$ we have $\Delta(\lambda)\equiv 1$. On the account of (\ref{nn90}), this means that 
	\beq\label{nnn91}
	\sum_{l=1}^n (-1)^l\sum_{1\le j_1< j_2<\dots< j_l\le n}e^{-\lambda(\tau_{j_1}+\tau_{j_2}+\dots+\tau_{j_l})} (P_0)_{j_1j_2...j_l}=0\quad \mbox{for all } \lambda\in \CC.
	\ee
	
	If $b_1$ and $b_2$ are distinct reals, then the functions  $e^{-\lambda b_1}$ and  $e^{-\lambda b_2}$
	of a complex variable  $\lambda$ are linearly independent. Appropriately choosing $\tau_l$
	for each  $ l\le n$, 
	we will show that the principal minors $(P_0)_{j_1j_2\dots j_l}$ are equal to zero for all $l\le n$. Set
	$
	\tau_0=1/(n+1).
	$
	Fix arbitrary  $1\le j_1< j_2<\dots< j_l\le n$. Put
	$\tau_{j_1}=\tau_{j_2}=\dots=\tau_{j_l}=1$ and   $\tau_i=\tau_0$ for all $i\ne j_k, k\le l$.
	Then one summand in  (\ref{nnn91}) is $(P_0)_{j_1j_2\dots
		j_l}e^{-l\lambda}$,
	while all other summands include factors of the type $e^{-\la r}$ with $r\ne l$.
	It follows that $(P_0)_{j_1j_2\dots j_l}=0$, as desired.
\end{proof}

\subsubsection{Proof of Theorem  \ref{s1}}\label{sec:last}
If the problem (\ref{ns0}) is superstable, then the resolvent
$R(\lambda; \A)$ of the infinitesimal operator $\A$
generated by (\ref{ns0}) is an entire function of a complex variable
$\lambda$. This follows from the resolvent formula \cite{Paz}
$$
R(\lambda; \A)x=\int_0^\infty\,e^{-\lambda t} T(t)x\,dt,\quad x\in L^2(0,1)^n,
$$
where $T(t)$ is a $C_0$-semigroup generated by the operator $\A$ and satisfying the estimate
$\|T(t)\|\le M(\ga)e^{-\ga t}$ for any $\ga>0$.
Consequently, $\sigma(\A)=\varnothing$.

On the other side, Theorem \ref{tm_2} says that the operator  $\A$ given by (\ref{A}), (\ref{AAA}) has empty spectrum if and only if the problem  (\ref{ns0}) is finite time stabilizable.

\begin{rem}
	\rm{
		As pointed out by an anonymous reviewer, Theorems \ref{tm_2}
		and \ref{s1} (autonomous setting) in the case $b_j\equiv 0$  (this involves no loss of generality) and $a_j>0$ could also  be derived using time-delay 
		equations. Indeed, the integration of the 
		system 	along the
		characteristic directions   results in the following functional-delay system;
		see also \cite[(3.16)]{cor1}:
		\begin{equation}
		\label{cor111}
		\phi_j(t)=\sum_{k=1}^n p_{jk}\phi_k(t-\tau_k),\quad i\le n,
		\end{equation}
		where $\phi_j(t)=u_j(0,t)$ and $\phi_j(\xi)$ for $\xi\in [-\tau_j,0]$ are known functions
		of the initial conditions~\reff{eq:in}.
		By a classical result on linear time-delay systems (see, e.g., 
		\cite[Theorem 3.5]{hale} and   \cite[(3.17)]{cor1}), system \reff{cor111} is exponentially stable if and only if there exists
		$\de > 0$ such that
		\beq\label{01}
		\Delta(\lambda)=0\mbox{ implies } \Re\la\le-\de,
		\ee
		where $ \Delta (\lambda) $ is given by  \reff{nn5} with the matrix
		$P=(p_{jk})$ in place of  $ P_0$; see also Lemma 3.1. 
		From \reff{01} it follows that the system \reff{cor111} (and, therefore, the
		original autonomous problem  \reff{eq:1}--\reff{eq:in}) is superstable if and only if
		\begin{equation}
		\label{02}
		\Delta(\lambda) \neq 0 \quad \mbox{for all} \, \, \lambda \in \CC.
		\end{equation}
		Under the condition \reff{02}, the FTS property can be obtained from appropriate estimates for the exponential decay as applied in 
		\cite{Coron} and \cite{coron2020finitetime}. 
		We are thankfull to the reviewer for drawing our attention to this approach.
	}	
\end{rem}

\subsection{Sufficient condition for robust FTS}

Due to the discussion in Subsection \ref{exk1}, the  spectral criterion is  not robust in the sense
of Definition \ref{def:robust}. To prove  robust FTS criteria stated in
Theorems  \ref{tm_5}, \ref{prl},  and \ref{prl1}, we first give a  sufficient  condition for robust FTS,
namely the condition \reff{lpr} in Lemma \ref{equiv1} below.

The regularity assumptions \reff{ss44} imposed on the coefficients of the system
\reff{eq:1} allow us  to put  the problem  \reff{eq:1}--\reff{eq:in}
into a smooth setting whenever the initial function $\vphi$
is sufficiently smooth.  We use  integration along characteristic curves: For given $j\le n$,
$x \in [0,1]$, and $t \ge 0$,
the $j$-th characteristic of \reff{eq:1}
passing through the point $(x,t)\in\Pi$ is defined
as the solution $\om_j(\xi)=\om_j(\xi,x,t)$ to the initial value problem  
$$
\partial_\xi\om_j(\xi, x,t)=\frac{1}{a_j(\xi,\om_j(\xi,x,t))},\;\;
\om_j(x,x,t)=t.
$$
The characteristic curve $\tau=\om_j(\xi,x,t)$ reaches the
boundary of $\Pi$ in two points with distinct ordinates. Let $x_j(x,t)$
denote the abscissa of that point whose ordinate is smaller.
Denote 
$$
c_j(\xi,x,t)=\exp \int_x^\xi
\left(\frac{b_{j}}{a_{j}}\right)(\eta,\om_j(\eta,x,t))\,d\eta
$$
and define  a linear operator 
$R:C(\Pi)^n\to C(\R_{+})^n$ by
\begin{eqnarray*}
	\displaystyle
	\left(Ru\right)_j(t) = \sum\limits_{k=1}^mp_{jk}u_k(1,t)+\sum\limits_{k=m+1}^np_{jk}u_k(0,t),
	\quad   j\le n.
\end{eqnarray*}

It is straightforward to show that a $C^1$-map $u:{\Pi}\to \R^n$
is a classical solution to
the problem \reff{eq:1}--\reff{eq:in}  if and only if
it 
satisfies the following system of functional equations:
\begin{equation}
\begin{array}{cc}
\label{Q}
u_j(x,t)=\left\{\begin{array}{lcl}
\displaystyle c_j(x_j(x,t),x,t) \left(Ru\right)_j(\om_j(x_j(x,t)))
& \mbox{if}&  
x_j(x,t)=0 \mbox{ or } x_j(x,t)=1 
\\
c_j(x_j(x,t),x,t)\vphi_j(x_j(x,t))      & \mbox{if}& 
x_j(x,t)\in(0,1).
\end{array}\right.
\end{array}
\end{equation}
A continuous function~$u$ satisfying  \reff{Q} in $\Pi$
is called a {\it continuous solution} to \reff{eq:1}--\reff{eq:in}.

Introduce a linear operator
$S: C(\R_{+})^n\to C(\Pi)^n$
by 
$$
(Sv)_j(x,t)=c_j(x_j(x,t),x,t)v_j(\om_j(x_j(x,t),x,t)), \quad j\le n.
$$
Due to  (\ref{Q}),  for given $l\in \N$,
the continuous solution to
\reff{eq:1}--\reff{eq:in}  satisfies  the equation  
\beq\label{eq:SR}
u(x,t)=\left[(SR)^lu\right](x,t)\quad\mbox{ for all } t> l/a.
\ee
Hence, the stabilization properties of
the problem under consideration are closely related to the powers of the linear operator 
$SR: C(\Pi)^n\to C(\Pi)^n$.

\begin{lemma}\label{equiv1} Let $n\ge 2$. Assume that
	\beq\label{lpr}
	p_{i_1i_2}p_{i_2i_3}\dots p_{i_{n}i_{n+1}}=0 \quad \mbox{for all tuples }
	\left(i_1,i_2,\dots,i_{n},i_{n+1}\right)\in\left\{1,\dots,n\right\}^{n+1}.
	\ee
	Then 	the problem \reff{eq:1}--\reff{eq:in} is robust finite time stabilizable.
\end{lemma}

\begin{proof} 
	Since \reff{lpr} does not depend on $a_j$ and $b_j$, then, on the account of \reff{eq:SR}, we are done if we prove that 
	\beq\label{C0}
	\left[(SR)^nu\right](x,t)\equiv 0\quad	\mbox{ for all } t> n/a,\,  u\in C\left(\Pi\right)^n.
	\ee
	
	Assume first  that   $n=2$.
	The condition   \reff{C0} in this case reads 
	\begin{eqnarray*}
		c_j(x_j,x,t) \left(RSRu\right)_j(\om_j(x_j,x,t))\equiv 0
		\quad\mbox{ for all } t>2/a,\,  u\in C\left(\Pi\right)^2, \mbox{ and }
		j\le 2,
	\end{eqnarray*}
	and  is satisfied whenever
	\begin{eqnarray}\label{*1}
	\left(RSRu\right)(t)\equiv 0
	\quad\mbox{ for all }  t>2/a  \mbox{ and } u\in C\left(\Pi\right)^2.
	\end{eqnarray}
	Here
	\beq
	\begin{array}{ll}
		\displaystyle \left(RSRu\right)_j(t)=\sum\limits_{k=1}^2p_{jk}(SRu)_k(1-x_k,t)
		\\\qquad
		\displaystyle=\sum\limits_{k=1}^2p_{jk}c_k(x_k,1-x_k,t)(Ru)_k(\om_k(x_k,1-x_k,t))
		\\\qquad
		\displaystyle=\sum\limits_{k=1}^2p_{jk}c_k(x_k,1-x_k,t)\sum\limits_{i=1}^2p_{ki}
		u_i(1-x_i,\om_k(x_k,1-x_k,t)),\label{*2}
	\end{array}
	\ee
	where we used the shorthand notation $x_j=x_j(x,t)$ for $j\le n$.
	At the same time, the condition   \reff{lpr}  for $n=2$ reads
	\beq\label{k6}
	p_{jk}p_{ki}=0\quad \mbox{ for all }   j,k,i\le 2.
	\ee
	The lemma now follows immediately from the equations \reff{*1}--\reff{k6}.
	
	The proof for $n=3$ is similar, with analogs of \reff{*1}, \reff{*2}, and \reff{k6} being
	\begin{eqnarray*}
		\left(RSRSRu\right)(t)\equiv 0
		\quad\mbox{ for all }  t>3/a  \mbox{ and } u\in C\left(\Pi\right)^3,
	\end{eqnarray*}
	\beq\label{re}
	\begin{array}{ll}
		\displaystyle\left(RSRSRu\right)_j(t)\\
		\qquad =\displaystyle\sum\limits_{k=1}^3p_{jk}c_k(x_k,1-x_k,t)
		\sum\limits_{i=1}^3p_{ki}
		c_i(x_i,1-x_i,\om_k(x_k,1-x_k,t))\\
		\quad\qquad\displaystyle \times
		\sum\limits_{s=1}^3p_{is}
		u_s(1-x_s,\om_i(x_i,1-x_i,\om_k(x_k,1-x_k,t))),
	\end{array}
	\ee
	and 
	$$
	p_{jk}p_{ki}p_{is}=0\quad \mbox{ for all }  j,k,i,s\le 3,
	$$
	respectively.
	
	It is clear that a similar argument works for any subsequent $n$.
\end{proof}

\begin{rem}\rm
	Lemma \ref{equiv1} proves the implication $(\io)\Rightarrow(\io\io)$ of Theorem \ref{prl1} 
	by showing that the condition \reff{lpr} is sufficient for the robust
	FTS property. In the {\it autonomous} setting, the proof of this lemma also suggests a possibility
	of proving the  inverse implication $(\io\io)\Rightarrow(\io)$ using the apparatus of delay equations.
	Assume, for simplicity, that $n=3$  and $b_{j}=0$ for $j\le 3$.
	According to \reff{eq:SR} and \reff{re}, the solution to the autonomous system
	restricted to $\d\Pi$ satisfies the following delay system:
	\beq\label{re0}
	\begin{array}{ll}
		\phi_j(t)
		\displaystyle=\sum\limits_{k=1}^3\sum\limits_{i=1}^3\sum\limits_{s=1}^3p_{jk}p_{ki}p_{is}
		\phi_s(t-\tau_i-\tau_k-\tau_j),\quad j\le 3,\  t>3/a,
	\end{array}
	\ee
	where $\phi_j(t)=u_j(x_j,t)$ and $\tau_j$ are defined by \reff{nn3}.
	This system could be used to derive the condition \reff{lpr} from the robust
	FTS property, completing the proof of the equivalence $(\io)\Leftrightarrow(\io\io)$ of Theorem \ref{prl1}.
	Note, however, that this way requires a certain amount of tedious
	work, mainly because the right hand side of \reff{re0}
	involves all functions $\phi_j$, $j\le 3$. We
	follow a different,
	combinatorial	approach as presented in the next subsection, which  allows us 
	to treat  the general {\it nonautonomous} setting.
\end{rem}

\subsection{Combinatorial  criterion: Proof of Theorem \ref{tm_5}} \label{sec:combin}

Recall that, given an   $n\times n$-matrix $P=(p_{jk})$, the graph $G_P$ was defined as the directed graph
with the adjacency matrix  $W=(\sgn|p_{jk}|)$.

{\it Sufficiency.}
Suppose that $G_P$ is acyclic. Consider an arbitrary sequence $j_1,\dots,\allowbreak j_{n+1}$, where $j_i\in\{1,\dots,n\}$.
It must contain two equal elements. Let $j_k=j_l$ for $k\le l$ such that the difference $l-k$ is minimum possible.
Since the subsequence $j_k,\dots,j_l$ does not form a cycle in $G_P$, there exists $s$ such that
$k\le s<l$ and $(j_s,j_{s+1})$ is not a directed edge of $G_P$, that is $p_{j_sj_{s+1}}=0$. It follows that
$ p_{j_1j_2}p_{j_2j_3}\cdots p_{j_{n}j_{n+1}}=0$ and, therefore, the condition \reff{lpr} is fulfilled.
By Lemma~\ref{equiv1}, the problem  (\ref{eq:1})--(\ref{eq:in}) is robust FTS. 

{\it Necessity.} Assume that the problem  (\ref{eq:1})--(\ref{eq:in}) is robust FTS.
By Corollary  \ref{sll1}, all principal minors of the matrix $P$
are equal to zero. The acyclicity of $G_P$ follows from the following lemma.

\begin{lemma}\label{cl8} $G_P$ is  acyclic 
	if and only if
	all principal minors of the matrix $P$
	equal  zero.
\end{lemma}

\begin{proof}  {\it Necessity.}
	Assume that $G_P$ is acyclic. We first show that $\det P=0$. Indeed,
	\beq\label{vanish0}
	\det P=\sum\limits_{\sigma}\sgn(\sigma)p_{1\sigma(1)}p_{2\sigma(2)}
	\cdots p_{n\sigma(n)}=\sum_{\sigma} s(\sigma),
	\ee
	the sum being taken over all
	permutations $\sigma$ of the set $\{1,2,\dots,n\}$.
	Consider a 	permutation~$\sigma$. This permutation decomposes
	into the product of independent
	cycles. Let $(j_1,\dots,j_l)$ be one of them. The sequence  $j_1,\dots,j_l$ does not form a cycle in $G_P$,
	which means that there exists $s\le l$ such that $(j_s,j_{s\oplus 1})$ is not a directed edge of $G_P$ , where $\oplus$ is addition modulo $l$,
	that is $p_{j_s,j_{s\oplus 1}}=0$.
	It follows that $ p_{j_1j_2}p_{j_2j_3}\cdots p_{j_{l}j_{1}}=0$.
	Therefore, $s(\sigma)=0$ for each $\sigma$, implying that $\det P=0.$
	
	To finish this part of the proof, it remains to note that every
	principal minor of  $P$  determines the adjacency matrix
	of a  subgraph of  $G_P$. As all subgraphs of a directed acyclic graph
	are  acyclic, the same argument as above implies that all principal minors of $P$ are equal to zero.
	
	{\it Sufficiency.} We prove that, if $G_P$ has a cycle, then $P$ has a nonzero principle minor.
	The proof is by induction on $n$.
	
	The base case $n=1$ is trivial: If $G_P$ is a cycle, then $w_{11}=1$ and, hence, $p_{11}\ne 0$.
	
	Assume that the claim is true for all matrices of order $1,\dots,n-1$. Given a matrix $P$ of order $n$,
	consider two cases.
	
	Case 1: $G_P$ has a cycle  $j_1,\dots,j_l$ of length $l<n$. Let $G^\prime$ be a subgraph of $G_P$ induced
	on the vertices $j_1,\dots,j_l$, and $P^\prime$ be the restriction of $P$ to the rows  and columns with indices
	$j_1,\dots,j_l$. Since $G^\prime$ has a cycle, the induction assumption implies that $P^\prime$
	has a nonzero principle minor, which is also a principle minor  of $P$.
	
	Case 2: $G_P$ has a cycle  $C$ of length $n$ and does not have any shorter cycles.
	Without loss of generality we can assume that $C$ is formed by the vertices $1,2,\dots,n$ in this order.
	Note that $C$ is  a unique cycle  of length $n$ in $G_P$. Indeed, any other cycle would contain an arrow $(i,j)$ absent in $C$. This is, however, impossible because this arrow would form a shorter cycle along with a part of $C$.
	
	Let us prove that $\det P\ne 0$.
	If   a permutation
	$\sigma$ in the equality \reff{vanish0} is not cyclic, then   $s(\sigma)=0$
	as in the necessity part. If $\sigma$ is cyclic but different from $(1,2,\dots,n)$,
	then $s(\sigma)=0$
	as well, because $s(\sigma)$ contains a factor $p_{ij}=0$ for an arrow $(i,j)$ not in $C$. Therefore,
	$\det P=s(\sigma)$ for $\sigma=(1,2,\dots,n)$, that is, 
	$\det P=\sgn(\sigma)p_{12}p_{23}\cdots p_{n1}\ne 0$. The proof is therewith complete.
\end{proof}

\begin{cor}\label{cor1}
	A constant zero-one  $n\times n$-matrix $W=(w_{jk})$ is the adjacency matrix
	of an acyclic directed graph  if and only if
	all principal minors of $W$ are equal to zero.
\end{cor}

\subsection{Algebraic  criterion} \label{sec:alcr}
\subsubsection{Proof of Theorem \ref{prl}}
Claims  $(\io)$ and $(iv)$ are equivalent accordingly to Theorem \ref{tm_5} and
Proposition \ref{prl0}. The equivalence of Claims  $(\io)$
and $(\io\io\io)$ follows from
Theorem \ref{tm_5} (see also Corollary \ref{cor1}).

To finish the proof, let us show the equivalence of Claims $(\io)$ and  $(\io\io)$. Assume to the contrary that
the problem  \reff{eq:1}--\reff{eq:in} is robust FTS and there exists a nonzero product of the type $w_{i_1i_2}w_{i_2i_3}\cdots w_{i_{n}i_{n+1}}$ where
$\left(i_1,i_2,\dots,i_{n},i_{n+1}\right)\in\left\{1,\dots,n\right\}^{n+1}$.
Since $W$ is an adjacency matrix of a directed graph $G_P$, then from the definition of an adjacency matrix we conclude
that then there is a cycle in $G_P$.
In other words, $G_P$ is not  acyclic, contradicting Theorem~\ref{tm_5}.

\subsubsection{Proof of
	Corollary~\ref{cortime}}
Similarly to the proof of  Lemma \ref{equiv1}, 
if  $W^{k_0}=0$, then
$$
((SR)^{k_0}u)(x,t)=0\quad  \mbox{for all}\,\,  t> k_0/{a_0} \mbox{ and } u\in C(\Pi)^n.
$$
Taking this into account,  Part $(\io)$ readily follows  from 
\reff{eq:SR} with $l=k_0$. 

Part $(\io\io)$ for $k_0=1$ straightforwardly follows from the 
solution formula \reff{Q} and the fact that  $p_{jk}=0$
for all $j\le n$ and $k\le n$ whenever $k_0=1$.

Now, let $k_0\ge 2$ and  prove that $T_{opt}\le T^*$.
We 
rewrite the problem \reff{eq:1}--\reff{eq:in} 
in an equivalent form. 
Since the coefficients $a_j$ do not depend on $t$, the characteristic curves 
are well defined for each $t\in\R$. 
We  extend the functions $b_j(x,t)$ 
by setting $b_j(x,t)\equiv 0$ for $t<0$. 
In the domain $\Pi$, consider 
the problem \reff{eq:1}, \reff{eq:3} with
the  delay boundary 
conditions
\beq\label{eq:inb}
u_j(x_j,t)=\vphi_j(\sigma_j(0,x_j,t)),\quad t\in[-\tau_j,0],\ j\le n,
\ee
where $\tau_j$ are defined by \reff{nn3}, $x_j=0$ for $j\le m$ and $x_j=1$ for $m+1<j\le n$,
and the map $\tau\mapsto \sigma_j(\tau,x_j,t)$ 
is the inverse  of the characteristic $\xi\mapsto\om_j(\xi,x_j,t)$. 

Note that the functions $c_j(x_j,x,t)$ are continuous in $(x,t)\in\Pi$. 
In the framework of the method of  characteristics, any continuous solution $u$ to the problem  \reff{eq:1}, \reff{eq:3},
\reff{eq:inb} in $\Pi$ is determined by the formula
\begin{equation}
\begin{array}{cc}
\label{Qb}
u_j(x,t)=\left\{\begin{array}{lcl}
\displaystyle c_j(x_j,x,t) \left(Ru\right)_j(\om_j(x_j,x,t))
& \mbox{if}&  \om_j(x_j,x,t)\ge 0
\\
c_j(x_j,x,t)\vphi_j(\sigma_j(0,x_j,\om_j(x_j,x,t)))
& \mbox{if}& \om_j(x_j,x,t)< 0,
\end{array}\right.
\end{array}
\end{equation}
where $j\le n$.
Since  $b_j\equiv 0$ for $t<0$, we have the obvious equality
$c_j(x_j,x,t)=c_j(x_j(x,t),x,t)$ for all $(x,t)\in\Pi$. Hence, the functions
$u_j$ given by \reff{Qb} do not depend on $b_i$ for $t<0$ and $i\le n$. 
Now, we conclude that 
the  formula  \reff{Qb} for continuous solutions to
the problem \reff{eq:1}, \reff{eq:3},
\reff{eq:inb}
is   an equivalent form of  \reff{Q}  for continuous solutions to
the problem \reff{eq:1}--\reff{eq:in}.

Suppose that $k_0=3$ (the same argument works for any other $k_0\ge 2$).
Let us introduce the following notation:
$$
\begin{array}{rcl}
y_j&=&1-x_j,\\ [1mm]
\displaystyle t_{jki}(x)&=&\displaystyle \int_{x_j}^x\frac{dx}{a_{j}(x)}+
\int_{0}^{1}\frac{dx}{|a_{k}(x)|}+\int_{0}^{1}\frac{dx}{|a_{i}(x)|},\\ [5mm]
\displaystyle c_{jki}(x,t)&=&\displaystyle c_j(x_j,x,t)\,c_k(x_k,y_k,\om_j(x_j,x,t))
\,c_i(x_i,y_i,\om_k(x_k,y_k,\om_j(x_j,x,t))).
\end{array}
$$
Notice that $c_{jki}(x,t)\ne 0$ for all $(x,t)\in\Pi$ and $j,k,i\le n$. 

To prove that $T_{opt}\le T^*$, it  suffices to show that for  $u$ given by \reff{Qb} it holds
$u_j\equiv 0$ for all $t\ge T^*$ and $j\le n$.
Let $t\ge T^*$ be arbitrary fixed. We have
$t-t_{jki}(x)\ge 0$ for all $(j,k,i)\in I$ and $x\in[0,1]$. For every
$j\le n$ and $x\in[0,1]$, the value $u_j(x,t)$ can be computed by iterating  three times the first line of the
formula~ \reff{Qb}. We, therefore, get
\beq\label{p00}
\begin{array}{ll}
	u_j(x,t)=\left[(SR)^3u\right]_j(x,t)
	=
	\displaystyle\sum_{	\substack{k,i,s=1\\ (j,k,i)\in I}}^np_{jk}p_{ki}p_{is}\,c_{jki}(x,t)\,
	u_s\left(x_s,t-t_{jki}(x)
	\right),
\end{array}
\ee
where all products $p_{jk}p_{ki}p_{is}$ equal zero by the assumption. Hence, 
$u_j(\cdot,t)\equiv 0$, as desired.

Finally, let $k_0=3$  and show that $T_{opt}=T^*$
in this case  (the same argument works  for $k_0=2$).
We are done if we prove  that  for any $\eps\in(0,\min_{j\le n}{\tau_j})$ there exists a function $\vphi$ and an index $j\le n$ such that $u_j(y_j,T^*-\eps)\ne 0$, where $u_j$ is given by \reff{Qb}.
Fix an arbitrary $\eps\in(0,\min_{j\le n}{\tau_j})$. 
Due to the definition of $T^*$, there exists a triple, say $(\al,\be,\ga)\in I$,
such that $T^*=t_{\al\be\ga}(y_\al)$.
Since
$t_{jki}(y_j)\le T^*$,
we have $T^*-\eps-t_{jki}(y_j)>-\min_{j\le n}{\tau_j}$
for all $(j,k,i)\in I$.
Moreover,
$T^*-\eps-t_{\al\be\ga}(y_\al)< 0.
$ This means that we can compute 
the value $u_\al(y_\al,T^*-\eps)$ by
iterating twice the  first line of the formula  \reff{Qb}   and  applying once the operator~$S$ of integration along characteristics.
Specifically, 
\beq\label{p0}
\begin{array}{l}
	u_\al(y_\al,T^*-\eps)
	=\left[(SR)^2Su\right]_\al(y_\al,T^*-\eps)\\
	\qquad\qquad\qquad\ \,	=
	\displaystyle\sum_{
		\substack{		k,i=1\\ (\al,k,i)\in I}}
	^np_{\al k}p_{ki}\,c_{\al ki}(y_\al,T^*-\eps)\,
	u_i(x_i,T^*-\eps-t_{\al ki}(y_\al)).
\end{array}
\ee
Consider an arbitrary pair $(k,i)$ such that
$(\al,k,i)\in I$ and
$T^*-\eps-t_{\al ki}(y_\al)\ge 0$. The value
$u_i(x_i,T^*-\eps-t_{\al ki}(y_\al))$
can be computed using the first line of  \reff{Qb}.
Doing so, we see that
the 
corresponding summand in \reff{p0} is equal to zero. 
If $(\al,k,i)\in I$ but
$T^*-\eps-t_{\al ki}(y_\al)< 0$, which includes the pair$(k,i)=(\be,\ga)$, then
the corresponding  summand in
\reff{p0} can be computed using the second line of  \reff{Qb}. 
We now choose $\vphi$ in \reff{Qb}
such that, for all triples  $(\al,k,i)$ of this kind except $(\al, \be,\ga)$
we have $\vphi_{i}(\sigma_i(0,x_\ga,T^*-\eps-t_{\al ki}(y_\al)))=0$,
while 
$\vphi_{\ga}(\sigma_\ga(0,x_\ga,T^*-\eps-t_{\al \be\ga}(y_\al)))\ne 0$.
This is possible, since, by our construction,  $0<\sigma_\ga(0,x_\ga,T^*-\eps-t_{\al \be\ga}(y_\al))<1$ and there is no  triple
$(\al,k,\ga)\in I$ such that $k\ne\be$.
This completes the proof.

\subsubsection{Proof of Theorem \ref{prl1}}
The equivalence of $(\io)$ and  $(\io\io)$ follows from Theorem \ref{prl} and the obvious property
$$
w_{i_1i_2}w_{i_2i_3}\cdots w_{i_{n}i_{n+1}}=0\ \mbox{  if and only if }\ p_{i_1i_2}p_{i_2i_3}\cdots p_{i_{n}i_{n+1}}=0.
$$

Assume that the problem  (\ref{eq:1})--(\ref{eq:in}) is robust FTS.
For given $i,j \le n$, the $ij$-th element $(P^n)_{ij}$ of the matrix $P^n$
is expressed by the formula
\begin{equation}\label{kl91}(P^n)_{ij}= \sum_{i_2=1}^n p_{ii_2}\sum_{i_3=1}^n p_{i_2i_3}
\dots\sum_{i_{n}=1}^n p_{i_{n}j}=
\sum_{i_2,i_3,\dots,i_{n}=1}^n p_{ii_2}p_{i_2i_3}\cdots p_{i_{n}j}.
\end{equation}
Due to Claim $(\io\io)$, all summands in the right hand side equal zero. In other words,
the matrix $P$ is nilpotent. This yields that the matrix $P_{abs}$ of the absolute values of
$p_{ij}$, namely $(P_{abs})_{ij}=|p_{ij}|$, is nilpotent as well.  The implication
$(\io)\Rightarrow(\io v)$ is, therefore, proved. The  implication
$(\io v)\Rightarrow(\io)$ follows from Claim $(\io\io)$,  the formula \reff{kl91},  and the property
that
$$
P_{abs}^n=0\  \ \mbox{  if and only if  }\  \ |p_{i_1i_2}||p_{i_2i_3}|\cdots |p_{i_{n}i_{n+1}}|=0 
$$
for all 
tuples  $\left(i_1,i_2,\dots,i_{n},i_{n+1}\right)\in\left\{1,\dots,n\right\}^{n+1}.$

Finally, combining Theorem \ref{tm_5} with Lemma \ref{cl8}, we conclude that
$(\io)$ and  $(\io\io\io)$ are equivalent. 

\subsubsection{Proof of Theorem \ref{ss22}}\label{nona}

We use  the following statement,
which is proved similarly to Lemma \ref{equiv1}.
\begin{lemma}\label{equiv2} Let $n\ge 2$,
	$W=(w_{jk})$ be a  constant zero-one $n\times n$-matrix, and 
	$q_{jk}\in C^1(\R_+)$, where 
	$j,k\le n$. 	
	Assume that 
	$$
	w_{i_1i_2}w_{i_2i_3}\cdots w_{i_{n}i_{n+1}}=0 \quad \mbox{for all tuples }
	\left(i_1,i_2,\dots,i_{n},i_{n+1}\right)\in\left\{1,\dots,n\right\}^{n+1}.
	$$
	Then 	the problem \reff{eq:1}--\reff{eq:in} with
	$p_{jk}=q_{jk}(t)w_{jk}$	
	is robust finite time stabilizable.
\end{lemma}
Since the matrix $W$ satisfies the condition 
$(\io\io)$ of Theorem \ref{prl},
the sufficiency part follows.  Since the problem is FTS for every $P(t)$, 
let us fix $P(t)$ such that $P(t)$ does not depend on $t$. The necessity part
follows now  from Theorem~ \ref{prl1}.

\begin{rem}
	The proof of Theorems \ref{prl},  \ref{prl1},  and \ref{ss22} use  robustness in an essential way. Note that the robust criteria are formulated in terms of the boundary coefficients only.   The FTS concept is technically  more complicated. Indeed, even in the autonomous case, 
	the  spectral criterion  shows that an FTS criterion
	can hardly be expected  to be formulated by means of the boundary coefficients only.
\end{rem}

\subsection*{Acknowledgements}
We thank Oleg Verbitsky for  graph-theoretic comments.

Irina Kmit was supported by the VolkswagenStiftung Project ``From Modeling and Analysis to Approximation''. 	Natalya Lyul'ko was supported by the state contract
of the 	Sobolev Institute of Mathematics, 
Project No. 0314-2019-0012.  


\end{document}